\documentclass[a4paper,11pt,english,reqno]{amsart}
\usepackage[utf8]{inputenc}
\usepackage{amsmath,amssymb}
\usepackage{color}
\usepackage{listings}
\usepackage{hhline}
\usepackage{multirow}

\usepackage{tikz}  
\usetikzlibrary{positioning}
\usepackage{ctable} 

\usepackage{hyperref}

\usepackage{pifont}% http://ctan.org/pkg/pifont
\newcommand{\xmark}{{\Large \ding{55}}}%

\usepackage{tabularray}

\usepackage{diagbox}

\usepackage{tikz-cd}
\usepackage{longtable}

\usepackage{url}
\usepackage{enumitem}
\usepackage{pdflscape}

\usepackage{hyperref}
\newtheorem{thm}{Theorem}
\newtheorem{cor}[thm]{Corollary}
\newtheorem{lem}[thm]{Lemma}
\newtheorem{prop}[thm]{Proposition}

\newcommand{\z}{\zeta_8}

\newcommand{\Z}{{\mathbb Z}}

\newcommand{\Q}{{\mathbb Q}}

\title[Quadratic points on the Fermat quartic over number fields]{Quadratic points on the Fermat quartic\\ over number fields}

\author{Enrique Gonz\'alez--Jim\'enez}
\address{Universidad Aut{\'o}noma de Madrid, Departamento de Matem{\'a}ticas, Madrid, Spain}
\email{enrique.gonzalez.jimenez@uam.es}

\thanks{The author is supported by Grant PID2022-138916NB-I00 funded by MCIN/AEI/10.13039/501100011033 and by ERDF A way of making Europe.}

     \date{\today}

\begin{document}

\begin{abstract}
Let $C$ be a curve defined over a number field $K$. A point $P\in C(\overline{\Q})$ is called $K$-quadratic if $[K(P):K]=2$. Let $K$ be a number field such that the rank of the elliptic curves $E_1:\,y^2= x^3 + 4x$ and $E_2:\,y^2= x^3 - 4x$ over $K$ are $0$. Under the above condition, we prove that the set of $K$-quadratic points on the Fermat quartic $F_4\colon X^4+Y^4=Z^4$ is finite and computable and we provide a procedure to compute this finite set. In particular, we explicitly compute all the $K$-quadratic points if $[K:\Q]<8$. Moreover, if the degree of $K$ is odd, we prove that all the $K$-quadratic points corresponds just to the $\Q$-quadratic points. 
\end{abstract}

\maketitle

\section{Introduction}\label{Sec1}

Fermat, in 1637, proved 
that the Fermat quartic $F_4\colon X^4+Y^4=Z^4$ has only trivial solutions (i.e $XY=0$). 
Since then, numerous mathematicians have studied the solutions of this equation over number fields. Let $K$ be a number field and denote by $F_4(K)$ the set of $K$-rational points of $F_4$. In this paper we aim to determine the set of $K$-quadratic points on $F_4$:
$$
\Gamma_2(F_4,K)=\bigcup\{F_4(L) \,:\,	K\subset  L\subset \overline{\Q}\,\,\mbox{and}\,\,[L:K]=2\}.
$$
Aigner \cite{A} gives $\Gamma_2(F_4,\Q)$, proving that $\Q(\sqrt{-7})$ is the only quadratic extension that contains nontrivial solutions (note that $(1 + \sqrt{-7})^4 + (1 - \sqrt{-7})^4 = 2^4$).
Ishitsuka et al. \cite[Theorem 7.3]{I}  give $\Gamma_2(F_4,\Q(\zeta_8))$, where $\zeta_8$ is a primitive $8^\text{th}$-root of unity. Recently, Tho \cite{tho}, using a method of Mordell \cite{M}, has obtained a different proof from that of Ishitsuka et al. After completing this work, we became aware that Khawaja and Jarvis \cite[\S 3, Theorem 6.4]{KJ} had previously studied the problem of determining the points on the Fermat quartic that lie in a quadratic extension of \( \Q(\sqrt{2}) \). In particular, they proved that all such points lie in one of the following fields $\Q(\sqrt{2}, \sqrt{-1})$, $\Q(\sqrt{2}, \sqrt{-7})$, $\Q(\sqrt[4]{2})$, or $\Q(\sqrt{-1} \sqrt[4]{2})$. We have noted that the method used is essentially the same as the one employed by Tho \cite{tho}, which serves as the basis of this article. Unlike previous articles, the present work uses the knowledge of the growth of the torsion of the elliptic curves $E_1:\,y^2= x^3 + 4x$ and $E_2:\,y^2= x^3 - 4x$ to obtain generalizations of the previous results.

In the present paper, we use Mordell’s approach to provide a procedure for computing $\Gamma_2(F_4,K)$ in cases where the set is finite. In particular, if the degree of $K$ is odd and $\Gamma_2(F_4,K)$ is a finite set we prove that $\Gamma_2(F_4,K)=\Gamma_2(F_4,\Q)$. Finally, we compute the finite set $\Gamma_2(F_4,K)$ explicitly if $[K:\Q]<8$.

Note that if $\overline{\Q}$ denotes an algebraic closure of $\Q$, then $F_4(\overline{\Q})$ has the following trivial points:
$$
\{[0:\pm 1:1], [0:\pm \zeta_8^2:1], [\pm 1: 0:1], [\pm \zeta_8^2:0:1]\}\cup \{ [1:\z^j:0]\,:\, j=1,3,5,7 \}.
$$
In particular, $[0:\pm 1:1], [0:\pm \zeta_8^2:1], [\pm 1: 0:1], [\pm \zeta_8^2:0:1]\in \Gamma_2(F_4,\Q)$, and $[1:\z^j:0]\in \Gamma_2(F_4,K)$, for $j=1,3,5,7$, if and only if $\Q(\sqrt{2})\subset K$ or $\Q(\sqrt{-1})\subset K$.

\

Suppose $Z\ne 0$. The change of variables $x=X/Z$ and $y=Y/Z$ gives the affine equation
$$
F_4^0\,:\, x^4+y^4=1.
$$
Note that if $K$ is a number field and $(x,y)\in F_4^0(K)$ then $(\pm x,\pm y), (\pm y,\pm x)\in F_4^0(K)$. We will say that $(x,y)\in F_4^0(K)$ is a primitive point of $F_4^0$ (or of $F_4$) if $(x,y)\notin F_4^0(K')$ for any proper subfield $K'\subset K$. Let denote by $S^0_2(F_4,K)$ the set obtained from $\Gamma_2(F_4,K)$ after removing all trivial points, non-primitive points, and keeping only one representative $(x,y)$ for each set $\{(\pm x,\pm y), (\pm y,\pm x)\}$.

\

Let be the elliptic curves $E_1:\,y^2= x^3 + 4x$ and $E_2:\,y^2= x^3 - 4x$. The main result of this paper is the following:
\begin{thm}\label{theorem}
Let $K$ be a number field such that $\operatorname{rank}_{\Z}E_j(K)=0$, $j=1,2$. Then:  
\begin{enumerate}
\item \label{t1} $\Gamma_2(F_4,K)$ is finite and computable.
{ \item \label{t2}If $[K:\Q]$ is odd, then $\Gamma_2(F_4,K)=\Gamma_2(F_4,\Q)$.}
\item  \label{t3} If $[K:\Q]<8$, then $\Gamma_2(F_4,K)=\Gamma_2(F_4,L)$, where $L\in\{\Q,\Q(\sqrt{-1}),\Q(\sqrt{2}),\Q(\z),\Q(\alpha)\}$ ($\alpha$ satisfies $\alpha^4-2\alpha^2-1=0$) is the biggest number field such that $L\subseteq K$. Moreover, the set $S^0_2(F_4,L)$ appears in the following table 
\renewcommand{\arraystretch}{1.5}
\begin{longtable}{|c|c|}
\hline
$L$ & $S^0_2(F_4,L)$ \\
 \hline 
\endfirsthead
\multicolumn{2}{c}%
{{\bfseries \tablename\ \thetable{} -- continued from previous page}} \\
\hline
$L$ & $S^0_2(F_4,L)$ \\
 \hline 
\endhead
\hline \multicolumn{2}{|r|}{{Continued on next page}} \\ \hline
\endfoot
\endlastfoot
$\Q$ & $(\omega,\overline{\omega})$\\
\hline

$\Q(\sqrt{-1})$& $\sqrt{-1}(\omega, \overline{\omega}),\quad (\omega, \sqrt{-1}\overline{\omega}), \quad (\sqrt{-1}\omega, \overline{\omega})$
\\ % $\zeta_8^2=\sqrt{-1}$ 
\hline 
$\Q(\sqrt{2})$ & $(1/\sqrt[4]{2},1/\sqrt[4]{2}),\quad \sqrt{-1}(1/\sqrt[4]{2},1/\sqrt[4]{2})$ \\ % \sqrt{2}= zeta8^3 - zeta
\hline 
$\Q(\alpha)$ & $(\delta ,-\alpha^3 + 2 \alpha ),\quad (\sqrt{-1}\delta ,-\alpha^3 + 2 \alpha )$\\
\hline 
\multirow{4}{*}{$\Q(\zeta_8)$}   
&
 $(\z, \sqrt[4]{2}),\quad (\z^3,\sqrt[4]{2}),\quad \z(1, \z\sqrt[4]{2}),\quad \z^2(\z, \sqrt[4]{2}), \quad
(1/\sqrt[4]{2},\z^2/\sqrt[4]{2}),$\\
&

$\zeta_3\zeta_8(1,\zeta_3),\quad \zeta_3\zeta_8^3(1,\zeta_3),\quad
\zeta_3\zeta_8(1,\zeta_3\zeta_8^2),\quad
\zeta_3\zeta_8(\zeta_8^2,\zeta_3),$\\
&
$
\omega^{-1}(\overline{\omega}\z,1),\quad 
\omega^{-1}(\overline{\omega}\z^3,1),\quad 
\omega^{-1}(\overline{\omega}\z,\z^2),\quad
\omega^{-1}(\overline{\omega}\z^3,\z^2) ,
$
\\
& 
 $
\overline{\omega}^{\,-1}({\omega}\z,1),\quad 
\overline{\omega}^{\,-1}({\omega}\z^3,1),\quad
\overline{\omega}^{\,-1}({\omega}\z,\z^2),\quad
\overline{\omega}^{\,-1}({\omega}\z^3,\z^2)
$\\
\hline
\multicolumn{2}{c}{$\zeta_3=e^{2i\pi/3}$, $\omega=(1+\sqrt{-7})/2$, $\alpha$ satisfies $\alpha^4-2\alpha^2-1=0$ and $\delta$ satisfies $\delta^2=\alpha^3 - 3 \alpha$.}
\end{longtable}
 \item  \label{t4} Let $L$ be a number field such that $E_i(K)=E_i(L)$ for $i=1,2$. Then  $\Gamma_2(F_4,K)=\Gamma_2(F_4,L)$.
\end{enumerate}
\end{thm}

The following corollaries are immediate consequences of the preceding theorem:

\begin{cor}
We have
$$
\renewcommand{\arraystretch}{1.5}
\begin{array}{|c|c|c|c|c|c|}
\hline
K & \Q & \Q(\sqrt{-1}) & \Q(\sqrt{2}) & \Q(\alpha) & \Q(\z)\\
\hline
|\Gamma_2(F_4,K)| & 16 & 44 & 28 & 44& 188\\ 
\hline
\end{array}$$
\end{cor}
\begin{cor}
Let $K$ be a number field. Then $F_4$ has non-trivial primitive points in a quadratic extension $L/K$ in the following cases:
\begin{enumerate}
\item If $K=\Q$ then $L=\Q(\sqrt{-7})$.
\item If $K=\Q(\sqrt{-1})$ then $L=\Q(\sqrt{-1}, \sqrt{-7})$.
\item If $K=\Q(\sqrt{2})$ then $L=\Q(\sqrt[4]{2})$ or $L=\Q(\sqrt{-1}\sqrt[4]{2})$.
\item If $K=\Q(\alpha)$ then $L=\Q(\alpha,\sqrt{ \alpha^3 - 3 \alpha})$ or $L=\Q(\alpha,\sqrt{- \alpha^3 + 3 \alpha})$.
\item If $K=\Q(\z)$ then $L=\mathbb{Q}\left(\zeta_8,\sqrt[4]{2}\right)$, $L=\mathbb{Q}\left( \zeta_8,\sqrt{-3}\right)$, or $L=\mathbb{Q}\left( \zeta_8,\sqrt{-7}\right)$.
\end{enumerate}
\end{cor}

In particular, (i) is Aigner's result and (v) is Ishitsuka et al.'s result.

\

 {\bf Remark:}  A straightforward consequence of Theorems~\ref{theorem}~\eqref{t4} is that one can compute \( \Gamma_2(F_4, K) \) for certain number fields \( K \) with \( [K : \Q] > 8 \). For instance, when \( K = \Q(\zeta_{16}) \), we have $\Gamma_2(F_4, \Q(\zeta_{16})) = \Gamma_2(F_4, \Q(\zeta_8))$, since $ E_i(\Q(\zeta_{16})) = E_i(\Q(\zeta_8)) $ for \( i = 1, 2 \).

\

This article is organized as follows. Section~\ref{quadratic} is devoted to the study of quadratic points on a curve. In particular, we describe an algorithm that, in practice, computes \( \Gamma_2(C, K) \) when the Jacobian of \( C \) has only finitely many \( K \)-rational points. In Section~\ref{modular}, we note that the Fermat quartic \( F_4 \) is isomorphic over \( \Q \) to the modular curve \( X_0(64) \), and we recall a method due to Özman and Siksek \cite{OS} to determine \( \Gamma_2(X_0(64), \Q) \). Although this method is, in principle, difficult to apply when replacing \( \Q \) with a number field \( K \), we explain in Section~\ref{algo} a technique inspired by an idea of Mordell that allows us to determine \( \Gamma_2(X_0(64), K) \), provided that \( E_1(K) \) and \( E_2(K) \) are finite. Section~\ref{sec_growth} contains another notable result of this paper, namely Proposition~\ref{prop2}, where we prove that if \( E \) is an elliptic curve defined over \( \Q \) with \( j(E) = 1728 \), then the torsion subgroup \( E(\Q)_{\mathrm{tors}} \) does not grow in any number field of odd degree. This result is of independent interest beyond the context of quadratic points on the Fermat quartic, and it plays a key role in the proof of Theorem~\ref{theorem}~\eqref{t2}. Finally, in Section~\ref{sec_proof}, we provide the proof of the main result of the article, namely Theorem~\ref{theorem}. The Appendix contains relevant data concerning certain genus one curves that are used throughout the article. We also list several number fields of degree less than 7 that satisfy the conditions in Theorem~\ref{theorem}.

This work makes extensive use of the computer algebra system \verb|Magma| \cite{MA}. The code verifying the computational claims made in this paper is available at \cite{repository}.

\section{Quadratic points on curves}\label{quadratic}
One of the fundamental results in the theory of quadratic points on curves is the following, which follows from results by  Abramovich and Harris \cite{AH} and Harris and Silverman \cite{HS} (see also \cite{Bars}).

 \begin{thm}\label{AHHS}
Let $C$ be a curve defined over a number field $K$ of genus $g\geq 2$. Then, $\Gamma_2(C,K)$ is an infinite set if and only if $C$ is hyperelliptic over $K$ or bielliptic over $K$ with a bielliptic map $C\rightarrow E$ such that the elliptic curve $E$ has positive rank over $K$. 
\end{thm}

For a non-hyperelliptic curve \( C\) of genus \( \geq 3 \) defined over a number field $K$ with finite \( J(K) \), where $J$ is the jacobian of C, and at least one point \( P_0\in C(K) \), a theoretical method exists to determine all quadratic points by computing effective degree $2$ rational divisors. This relies on the injective map \( \iota: C^{(2)}(K) \to J(K) \), where $C^{(2)}$ denote the second symmetric product of $C$, which enables the recovery of quadratic points from elements of \( J(K) \).  For each \( [D] \in J(K) \), one computes the Riemann-Roch space \( \mathcal{L}(D + 2P_0) \). If its dimension is $1$, an effective divisor of degree 2 can be determined. However,  computing \( J(K) \) is often difficult. Even when feasible, the required Riemann-Roch computations may be impractical for large groups.

Thus, while theoretically sound, this approach is computationally demanding and may not always be practical.

In Section \ref{algo}, we present a much simpler method than the previous one, based on Mordell's ideas, which allows us to determine the quadratic points of the Fermat quartic over a number field.

\section{A modular approach}\label{modular}

It is noteworthy that the Fermat quartic is $\Q$-isomorphic to the non-hyperelliptic genus $3$ modular curve $X_0(64)$ and  $E_1$ is $\Q$-isomorphic to the modular curve $X_0(32)$.

Ozman and Siksek \cite{OS} determine the quadratic points on the modular
curves $X_0(N)$, where the curve is non-hyperelliptic,
the genus is $3$, $4$ or $5$, and the Mordell--Weil group
of $J_0(N)$ is finite. In particular the case of the modular curve $X_0(64)$. Their method involves first determining the rational cuspidal subgroup \(  C_0(64)(\mathbb{Q}) \) (see \cite{OS} for definition) and bound its index in \( J_0(64)(\mathbb{Q}) \), where \( J_0(64) \) is the jacobian of $X_0(64)$. This provides a positive integer \( I \) such that \( I \cdot J_0(64)(\mathbb{Q}) \subseteq C_0(64)(\mathbb{Q}) \). Consequently, the effective degree $2$ divisors \( D \) satisfy \( [D - 2P_0] = I \cdot [D^\prime] \) for some \( [D^\prime] \in J_0(64)(\mathbb{Q}) \).  To refine the search, they apply a version of the Mordell–Weil sieve to eliminate most possibilities for \( D^\prime \). Finally, they use Riemann–Roch theory to explicitly determine the divisors \( D \).  

It would seem natural to apply this method by replacing $\Q$ by a number field $K$ which $J_0(64)(K)$ is finite. Although theoretically this seems possible, the calculations involved in number fields lead us to believe that achieving our goal would not be easy, assuming everything worked as expected. For this reason, we propose a much simpler method in Section \ref{algo} that works efficiently.

\begin{lem}\label{ranklema}
We have $J_0(64)\stackrel{\Q}{\sim} E_1^2\times E_2$. In particular, if  $K$ is a number field, then $\operatorname{rank}_{\Z}J_0(64)(K)=2 \operatorname{rank}_{\Z}E_1(K)+\operatorname{rank}_{\Z}E_2(K)$
\end{lem}

\begin{proof}
The following table shows, in the first column, three involutions in $F_4$ along with their corresponding definitions. Finally, the last column presents the quotient elliptic curve obtained from these involutions.
$$
\renewcommand{\arraystretch}{1.5}
\begin{array}{|c|r|c|}
\hline
\phi & \phi([X:Y:Z]) & F_4/\langle \phi \rangle \\
\hline
\phi_1 & [-X:Y:Z] & E_1 \\
\hline
\phi_2 & [X:-Y:Z] & E_1 \\
\hline
\phi_3 & [X:Y:-Z] & E_2 \\
\hline
\end{array}
$$
Since the genus of $X_0(64)$ is $3$, the above table shows
$$
J_0(64)\stackrel{\Q}{\sim} E_1^2\times E_2.
$$
In particular, $\operatorname{rank}_{\Z}J_0(64)(K)=2 \operatorname{rank}_{\Z}E_1(K)+\operatorname{rank}_{\Z}E_2(K)$
\end{proof}
\section{On torsion growth of elliptic curves.}\label{sec_growth}
The following result is of independent interest. It will directly lead to the deduction of the second part of Theorem \ref{theorem}.

\begin{prop}\label{prop2}
Let $E$ be an elliptic curve defined over $\Q$ with $j(E)=1728$ and $K$ a number field of odd degree. Then $E(K)_{\text{tors}}=E(\Q)_{\text{tors}}$.
\end{prop}

\begin{proof}
First of all, notice that since $j(E)=1728$, $E$ is $\Q$-isomorphic to $E_k\,:\,y^2=x^3+kx$ for some fourth-power-free integers $k$ (see \cite[Appendix A \S3]{Sil2}). It is well known that (see \cite[Prop. 6.1(a)]{Sil}): 
$$
E_k(\Q)_{\text{tors}}\simeq 
\left\{
\begin{array}{cll}
\Z/4\Z & &  \mbox{if $k=4$},\\
\Z/2\Z \times \Z/2\Z & &  \mbox{if $-k$ is a perfect square},\\
\Z/2\Z & &  \mbox{otherwise}.
\end{array}
\right.
$$
Let us prove that if \( K \) is a number field such that \( E(K)_{\text{tors}} \ne E(\Q)_{\text{tors}} \), then its degree is even. To this end, it suffices to prove that if \( P = (x,y) \in E(\overline{\Q})_{\text{tors}} \) with \( P \notin E(\Q) \), then \( [\Q(P) : \Q] \) is even, where $\Q(P)=\Q(x,y)$. Specifically, it is sufficient to consider \( P \) of odd prime order \( p \) or of a power of \( 2 \) order .

Let $p$ be a prime and $G_E(p)$ be the image (up to conjugacy into $\operatorname{GL}_2(\mathbb F_p))$ of the mod $p$ Galois representation on the $p$-torsion of $E$. In the complex multiplication case, the possibilities for $G_E(p)$ are completely understood; thanks of the work of Zywina \cite[\S 1.9]{Z}. Since $j(E)=1728$, $E$ has complex multiplication by $\Z[\sqrt{-1}]$ and we have
\begin{itemize}
    \item $p=2$, then $G_E(p)$ is $B(2)$ or $\text{C}_s(2)$,
    \item $p\equiv 1\pmod{4}$, then $G_E(p)$ is $\text{C}^+_{ns}(p)$ or $\text{GL}_2(\mathbb F_p)$,
    \item $p\equiv 3\pmod{4}$, then $G_E(p)$ is $\text{C}^+_s(p)$ or $\text{GL}_2(\mathbb F_p)$,
\end{itemize}
where $B(p)$ (resp. $\text{C}_s(p)$, $\text{C}_{ns}(p)$) denotes the Borel (resp. split Cartan, non-split Cartan) subgroup of $\text{GL}_2(\mathbb F_p)$ and $\text{C}^+_s(p)$ (resp. $\text{C}^+_{ns}(p)$) is the normalizer of $\text{C}_s(p)$ (resp. $\text{C}_{ns}(p)$) in $\text{GL}_2(\mathbb F_p)$.

Let $p$ be a prime and $P\in E[p]$ be a point of order $p$ in the elliptic curve $E$. Then \cite[Theorem 5.6]{GJN1} gives the possible degrees $[\Q(P) : \Q]$ depending on $G_E(p)$. In the particular case $j(E)=1728$ we have
\begin{itemize}
    \item $p=2$, then $[\Q(P) : \Q]\in\{1,2\}$,
    \item $p\equiv 1\pmod{4}$, then $[\Q(P) : \Q]=p^2-1$,
    \item $p\equiv 3\pmod{4}$, then $[\Q(P) : \Q]\in\{p^2-1,2(p-1),(p-1)^2\}$.
\end{itemize}
Therefore, if $p$ is an odd prime, then  $[\Q(P) : \Q]$ is even. 

To finish the proof, let $P$ be a point of order a power of $2$ not defined over $\Q$. Thanks to \cite[Theorem 4.6]{GJN1} we see that $[\Q(P) : \Q(2P)]$ divides $2$. In particular, this proves that $[\Q(P) : \Q]$ is even, since the points of order $2$ in $E$ are defined over $\Q$ or over a quadratic number field.
\end{proof}

\section{A Procedure Inspired by Mordell's approach} \label{algo}
In this section, giving a number field $K$, we provide a procedure that determines the set $S^0_2(F_4,K)$ once the sets $E_1(K)$ and $E_2(K)$ are known.
 
Let $L=K(\sqrt{d})$ be a quadratic extension of $K$, where $d\in K$ with $d\not\in K^2$. Since $x\neq 0$, $y^2\neq \pm 1$, and $x^4+y^4=1$, we have $(1+y^2)/x^2=x^2/(1-y^2)$.  Let 
\begin{equation}\label{PT1}
 t=\dfrac{1+y^2}{x^2}\left(=\dfrac{x^2}{1-y^2}\right).
 \end{equation}
It is a straightforward computation to check that $t\not\in \{0, \pm 1, \pm \zeta_8^2\}$. It directly follows from $x^4+y^4=1$ and \eqref{PT1} that
 \begin{equation}\label{EQ4}
 x^2=\dfrac{2t}{t^2+1},\,y^2=\dfrac{t^2-1}{t^2+1}.\end{equation}
Let $s=2t$, $2u=x(s^2+4)$, and $2v=xy(s^2+4)$. We deduce from \eqref{EQ4} that 
\begin{equation}\label{E1E2}
\begin{array}{l}
E_1\,:\,u^2=s^3+4s,\\
E_2\,:\,v^2=s^3-4s.
\end{array}
\end{equation}
We are going to split our procedure into two parts depending on whether $t\in K$.

\noindent $\bullet$ {\bf Step I: ${t\in K}$}. Let $u=a_1+b_1\sqrt{d}$ and $v=a_2+b_2\sqrt{d}$, where $a_1,a_2,b_1,b_2\in K$. From \eqref{E1E2}, $u^2,v^2\in K$. Since $\sqrt{d}\not\in K$,  $a_1b_1=a_2b_2=0$. We consider the following cases:
\begin{itemize}
\item $b_1=0$. Then $u=a_1\in K$.  Thus $(s,u)\in E_1(K)$. 
\item $b_2=0$. Then $v=a_2\in K$.  Thus $(s,v)\in E_2(K)$. 
\item $b_1 b_2\ne 0$. Then $a_1=a_2=0$. It follows from \eqref{E1E2} that
$$
\begin{array}{l}
db_1^2=s(s^2+4),\\
db_2^2=s(s^2-4).
\end{array}
$$
Hence, 
$$
H_3\,:\, r^2=(s^2+4)(s^2-4),
$$
where $r=db_1b_2v/s\in K$. Thus $(s,r)\in H_3(K)$. Note that the curve $H_3$ is $\mathbb Q$-isomorphic to $E_1$ by the map: \[(x,y)\in E_1\mapsto \left(\dfrac{2(x+2)}{x-2},\dfrac{16 y}{(x - 2)^2}\right)\in H_3.\]
\end{itemize}
  In this part of the procedure, we need to compute the $K$-rational points of the curves \( E_1 \), \( E_2 \), and \( H_3 \). From the first coordinate of each point, we obtain \( s \), and subsequently calculate \( x^2 \) and \( y^2 \), and finally calculate $(x,y)\in S^0_2(F_4,K)$.
  
\noindent $\bullet$ {\bf Step II: ${t\notin K}$}. Let $Q(T)\in K[T]$ be the monic minimal polynomial of $s$ over $K$. Since $[L:K]=2$ and $t\notin K$, then $\mathrm{deg}Q(T)=2$ and $L=K(s)$. By \eqref{E1E2}, there exist $\alpha_1,\alpha_2 ,\beta_1,\beta_2 \in K$ such that
$$
\begin{array}{rcl}
s^3+4s&=&(\alpha_1+\beta_1 s)^{2}, \\
s^3-4s&=&(\alpha_2+\beta_2 s)^{2}. 
\end{array}
$$
Thus, there exist $r_{1}, r_{2} \in K$ such that
$$
\begin{array}{rcl}
T^3+4T-(\alpha_1+\beta_1 T)^{2}&=&Q(T)\left(T-r_{1}\right),\\
T^3-4T-(\alpha_2+\beta_2 T)^{2}&=&Q(T)\left(T-r_{2}\right).
\end{array}
$$
Hence, $\left(r_{j}, \alpha_j+\beta_j r_{j}\right)\in E_j(K)$ for $j=1,2$. For  $P_j=(x_j,y_j)\in E_j(K)$, let $Q_{\beta_{j}}(T)\in K[T]$, $j=1,2$, such that 
$$
\begin{array}{rcl}
T^3+4T-(y_1-\beta_1 x_1+\beta_1 T)^{2}&=&Q_{\beta_1}(T)\left(T-x_{1}\right),\\
T^3-4T-(y_2-\beta_2 x_2+\beta_2 T)^{2}&=&Q_{\beta_2}(T)\left(T-x_{2}\right).
\end{array}
$$
Note that for $j=1,2$, replacing $P_j$ by $-P_j$ is equivalent to changing $\beta_j$ by $-\beta_j$. Therefore, we only need to implement the computations with one element in each set $\{\pm P_j\}$.

Finally, we find $\beta_1,\beta_2\in K$ by solving the system of quadratic equations coming from the equality $Q_{\beta_1}(T)=Q_{\beta_2}(T)$, 
and then $Q(T)$, the minimal polynomial of $s=t/2$ over K. That is, we obtain \( s \), and subsequently calculate \( x^2 \) and \( y^2 \), and finally $(x,y)\in S^0_2(F_4,K)$.

\section{Proof of Theorem \texorpdfstring{\ref{theorem}}{}}\label{sec_proof}

\subsection*{Proof of Theorem \texorpdfstring{\ref{theorem} \eqref{t1}}{}} In Section \ref{algo}, given a number field \( K \), we provided a procedure that allows the computation of \( \Gamma_2(F_4,K) \) from the points of \( E_1(K) \) and \( E_2(K) \). In the particular case where $\operatorname{rank}_{\Z}E_j(K)=0$, $j=1,2$, these sets are finite, and thus \( \Gamma_2(F_4,K) \) is also finite. This proves the first part of Theorem \ref{theorem}.

A different proof comes from the Theorem \ref{AHHS} along with the decomposition of the Jacobian of $F_4$ established in Lemma \ref{ranklema}.
\subsection*{Proof of Theorem \texorpdfstring{\ref{theorem} \eqref{t2}}{}}
Notice that $j(E_1)=j(E_2)=1728$, therefore if \( K \) is a number field of odd degree, Proposition \ref{prop2} asserts that $E_j(K)_{\text{tors}}=E_j(\Q)_{\text{tors}}$, $j=1,2$. Moreover, in the particular case where \( \text{rank}_\mathbb{Z} E_j(K) = 0 \) we have $E_j(K)=E_j(\Q)$, $j=1,2$. Theorem \ref{theorem}\eqref{t2} follows from the procedure described in Section \ref{algo}. That is, \( \Gamma_2(F_4,K)=\Gamma_2(F_4,\Q) \).

\subsection*{Proof of Theorem \texorpdfstring{\ref{theorem} \eqref{t3}}{}}
Let $K$ be a number field such that $\operatorname{rank}_{\Z}E_j(K)=0$, $j=1,2$. Then the procedure given in Section \ref{algo} allows the computation of \( \Gamma_2(F_4,K) \) from the points of \( E_1(K)_\text{tors} \) and \( E_2(K)_\text{tors} \). In particular, if $L$ is the maximal subfield contained in $K$ such that $E_j(K)_\text{tors}=E_j(L)_\text{tors}$, for $j=1,2$, then \( \Gamma_2(F_4,K)= \Gamma_2(F_4,L) \). This restricts the study to number fields in which the torsion grows in \( E_1 \) and/or \( E_2 \). In the appendix, it is shown that such fields are $\Q(\sqrt{-1})$, $\Q(\sqrt{2})$, $\Q(\zeta_8)$ and $\Q(\alpha)$, where $\alpha$ satisfies $\alpha^4-2\alpha^2-1=0$. Note that for those fields the rank of the elliptic curves $E_1$ and $E_2$ are $0$. Therefore, it suffices to compute \( \Gamma_2(F_4,K) \) for those fields and \( \Gamma_2(F_4,\Q)\). For this purpose, for each of those number fields $K$, the required data for these points are provided in the appendix.

{\bf Step I.} The affine points in $E_1(K), E_2(K)$, and $H_3(K)$. 

{\bf Step II.} We give a table where for each pair of points $(P_1,P_2)$, $P_1\in E_1(K)$ and $P_2\in E_2(K)$, we show the symbol \xmark \,\, if there does not exist solutions $(\beta_1,\beta_2)$ in $K$ or the corresponding quadratic polynomial $Q(x)$ is reducible or $t\in \{0, \pm 1, \pm \sqrt{-1}\}$. Otherwise, we show the corresponding irreducible quadratic minimal polynomial $Q(x)$ of $s$ and the point in $S^0_2(F_4,K))$.

\subsection{\texorpdfstring{$\Gamma_2(F_4,\Q)$}{}} Then $K=\Q$ and
$$
\begin{array}{l}
E_1(\Q)=\{(0,0),(2,\pm 4),[0:1:0]\},\\
E_2(\Q)=\{(0,0),(\pm 2,0),[0:1:0]\},\\
H_3(\Q)=\{(\pm 2,0),[1:\pm 1:0]\}.
\end{array}
$$

{\bf Step I.}  For each value of $s$ we obtain the following points in $\Gamma_2(F_4,\Q)$:
$$
\renewcommand{\arraystretch}{1.5}
\begin{array}{|c|c|c|}
\hline 
s & (x,y)\in \Gamma_2(F_4,\Q)) & \text{Minimal polynomial of $b$ over $\Q$}\\
\hline
 0 & (0 , b )& T^2 + 1\\ 
 \hline
 2& (1 , 0 )& T-1 \\ %T^4 - 2 T^2 - 1\\ 
 \hline
 -2& (b , 0 )& T^2 + 1\\ 
 \hline
\end{array}
$$
Note that those points in  $\Gamma_2(F_4,\Q)$ are trivial points. 

\

\noindent {\bf Step II.} We determine the following table 
\begin{center}
\renewcommand{\arraystretch}{1.5}
\begin{longtable}{|c|c|c|}
\hline
\backslashbox{$P_2$}{$P_1$} & $(0,0)$ & $(2,4)$\\
\hline 
%----------------------------------------------
$(0,0)$ & \xmark & \xmark \\
\hline 
%----------------------------------------------
$(2,0)$ &  \xmark &  $\begin{array}{c} (-s - 1 , -s)\\ x^2+x+2  \end{array}$ \\
\hline 
%----------------------------------------------
$(-2,0)$  &  \xmark   & \xmark\\
\hline 
\end{longtable}  
\end{center}
%----------------------------------------------
Therefore, $S^0_2(F_4,\Q)=\{(\omega,\overline{\omega})\}$, where $\omega=(1+\sqrt{-7})/2$ and $\overline{\omega}=(1-\sqrt{-7})/2$.
 
This finishes the proof for the computation of $\Gamma_2(F_4,K)$. In particular, we have showed  that if $F_4$ has non-trivial points in a quadratic number field $\Q(\sqrt{d})$, then $d=-7$, which is Aigner's result (cf. \cite{A}). 

\subsection{\texorpdfstring{$\Gamma_2(F_4,\Q(\sqrt{-1}))$}{}} Then $K=\Q(\sqrt{-1})$ and
$$
\begin{array}{l}
E_1(\Q(\sqrt{-1}))=E_1(\Q)\cup\{(-2, \pm 4 \sqrt{-1}), (\pm 2 \sqrt{-1}, 0)\},\\
E_2(\Q(\sqrt{-1}))=E_2(\Q),\\
H_3(\Q(\sqrt{-1}))=H_3(\Q)\cup \{(0 , \pm 4\sqrt{-1} ), (\pm 2\sqrt{-1} , 0 )\}.
\end{array}
$$
{\bf Step I.}  For each value of $s\notin \Q$ we compute the following points in $\Gamma_2(F_4,\Q(\sqrt{-1}))$:
$$
\renewcommand{\arraystretch}{1.5}
\begin{array}{|c|c|c|}
\hline 
s & (x,y)\in \Gamma_2(F_4,\Q(\sqrt{-1})) & \text{Minimal polynomial of $b$ over $\Q$}\\
\hline
 2\sqrt{-1}  & (-b, 1) &  T^2 + \sqrt{-1}\\
 \hline
 -2\sqrt{-1} & (-b, 1) & T^2 - \sqrt{-1}\\
 \hline
\end{array}
$$
Note that those points in $\Gamma_2(F_4,\Q(\sqrt{-1}))$ are trivial points. 

\

\noindent {\bf Step II.} We calculate the following table 
\begin{center}
\setlength\LTleft{-10mm}
\renewcommand{\arraystretch}{1.5}
\begin{longtable}{|c|c|c|c|c|c|}
\hline
\backslashbox{$P_2$}{$P_1$} & $(0,0)$ & $(2,4)$ & $(-2, 4 \sqrt{-1})$ & $(2 \sqrt{-1}, 0)$ & $(-2 \sqrt{-1}, 0)$\\
\hline 
%----------------------------------------------
$(0,0)$ & \xmark & \xmark & \xmark & \xmark & \xmark \\
\hline 
%----------------------------------------------
$(2,0)$ & \xmark & $\begin{array}{c} (-s - 1 , -s)\\ x^2+x+2  \end{array}$  &  $\begin{array}{c} (-1/2\sqrt{-1}s , -1/2\sqrt{-1}s + \sqrt{-1}) \\ x^2-2 x+8 \end{array}$  & \xmark & \xmark \\
\hline
%----------------------------------------------
$(-2,0)$ & \xmark & $\begin{array}{c}
(-1/2 s , -1/2\sqrt{-1} s - \sqrt{-1}) \\
x^2+2 x+8 \\ \end{array}$ & $\begin{array}{c}  
(-\sqrt{-1}s + \sqrt{-1} , -s)\\
x^2- x+2  \end{array}$ & \xmark & \xmark \\
 \hline
\end{longtable}  
\end{center}
Therefore, $S^0_2(F_4,\Q(\sqrt{-1}))=\{\sqrt{-1} (\omega,\overline{\omega}), 
 (\sqrt{-1} \omega,\overline{\omega}), (\omega,\sqrt{-1}\overline{\omega})
\}$.
This finishes the proof for the computation of $\Gamma_2(F_4,\Q(\sqrt{-1}))$.

\subsection{\texorpdfstring{$\Gamma_2(F_4,\Q(\sqrt{2}))$}{}} Then $K=\Q(\sqrt{2})$ and
$$
\begin{array}{l}
E_1(\Q(\sqrt{2}))=E_1(\Q),\\
E_2(\Q(\sqrt{2}))=E_2(\Q)\cup \{
\pm(2(\sqrt{2} + 1), 4(\sqrt{2} +1), 
\pm (2(-\sqrt{2} + 1), 4(\sqrt{2} - 1))
\},\\
H_3(\Q(\sqrt{2}))=H_3(\Q).
\end{array}
$$
{\bf Step I.}  For each value of $s\notin \Q$, we obtain the following points in $\Gamma_2(F_4,\Q(\sqrt{2}))$:
$$
\renewcommand{\arraystretch}{1.5}
\begin{array}{|c|c|c|}
\hline 
s & (x,y)\in \Gamma_2(F_4,\Q(\sqrt{2})) & \text{Minimal polynomial of $b$ over $\Q$}\\
\hline
2\sqrt{2} + 2 & (-b , -b )  & T^2 - 1/2\sqrt{2}\\
\hline
-2\sqrt{2} + 2 &  (-b , -b ) & T^2 + 1/2\sqrt{2} \\
 \hline
\end{array}
$$
We determine the points $(1/\sqrt[4]{2},1/\sqrt[4]{2}),\sqrt{-1}(1/\sqrt[4]{2},1/\sqrt[4]{2})$ in $S^0_2(F_4,\Q(\sqrt{2}))$.
\newpage
\noindent {\bf Step II.} We compute the following table 
\begin{center}
\renewcommand{\arraystretch}{1.5}
\begin{longtable}{|c|c|c|}
\hline
\backslashbox{$P_2$}{$P_1$} & $(0,0)$ & $(2,4)$\\
\hline 
%----------------------------------------------
$(0,0)$ & \xmark & \xmark \\
\hline 
%----------------------------------------------
$(2,0)$ & \xmark
&   $\begin{array}{c} (-s - 1 , -s)\\ x^2+x+2  \end{array}$  \\
\hline 
%----------------------------------------------
$(-2,0)$  &  \xmark  & \xmark\\
\hline 
%----------------------------------------------
$(2(\sqrt{2} + 1), 4(\sqrt{2} +1))$ & \xmark & \xmark\\
\hline
$(2(-\sqrt{2} + 1), 4(-\sqrt{2} + 1))$ & \xmark & \xmark\\
\hline
\end{longtable}  
\end{center}
%----------------------------------------------
Therefore, $S^0_2(F_4,\Q(\sqrt{2})))=\{(1/\sqrt[4]{2},1/\sqrt[4]{2}),\sqrt{-1}(1/\sqrt[4]{2},1/\sqrt[4]{2})
\}$.
This finishes the proof for the computation of $S^0_2(F_4,\Q(\sqrt{2}))$.

\subsection{\texorpdfstring{$\Gamma_2(F_4,\Q(\alpha))$}{}} 
Then $\sqrt{2}=1-\alpha^2$, $K=\Q(\alpha)$,  and 
$$
\begin{array}{l}
E_1(\Q(\alpha))=E_1(\Q)\cup  \{\pm S_{1,1},\pm S_{1,2}\},\\
E_2(\Q(\alpha))=E_2(\Q(\sqrt{2})),\\
H_3(\Q(\alpha))=H_3(\Q)  \cup  \{(\pm 2\alpha , \pm 4(\alpha^3 - \alpha))\}.
\end{array}
$$
{\bf Step I.}  For each value of $s\notin\Q(\sqrt{2})$ we determine the following points in $\Gamma_2(F_4,\Q(\alpha))$:
$$
\renewcommand{\arraystretch}{1.5}
\begin{array}{|c|c|c|}
\hline 
s & (x,y)\in \Gamma_2(F_4,\Q(\alpha))) & \text{Minimal polynomial of $b$ over $\Q(\alpha)$}\\
\hline
  -2 \alpha& (b , -\alpha^3 + 2 \alpha )& T^2  -\alpha^3 + 3 \alpha\\ 
 \hline
  -2 \alpha^3 + 2 \alpha^2 + 2 \alpha &  (-\alpha^3 + 2 \alpha , b) &  T^2  -\alpha^3 + 3 \alpha\\ 
 \hline
 2 \alpha & (b , -\alpha^3 + 2 \alpha ) &  T^2 + \alpha^3 - 3 \alpha\\
 \hline
 2 \alpha^3 + 2 \alpha^2 - 2 \alpha &  (-\alpha^3 + 2 \alpha , b) &  T^2 + \alpha^3 - 3 \alpha\\ 
 \hline
\end{array}
$$
We obtain the points $(\delta ,-\alpha^3 + 2 \alpha ), (\sqrt{-1}\delta ,-\alpha^3 + 2 \alpha )$ in $S^0_2(F_4,\Q(\alpha))$, where $\delta$ satisfies $\delta^2=\alpha^3 - 3 \alpha$.

\noindent {\bf Step II.} We compute the following table 
\begin{center}
\renewcommand{\arraystretch}{1.5}
\begin{longtable}{|c|c|c|c|c|}
\hline
\backslashbox{$P_2$}{$P_1$} & $(0,0)$ & $(2,4)$ & $S_{1,1}$ & $S_{1,2}$\\
\hline 
%----------------------------------------------
$(0,0)$ & \xmark & \xmark & \xmark &\xmark\\
\hline 
%----------------------------------------------
$(2,0)$ & \xmark
&  $\begin{array}{c} (-s - 1 , -s)\\ x^2+x+2  \end{array}$ & \xmark  &\xmark\\
\hline 
%----------------------------------------------
$(-2,0)$  & \xmark    & \xmark & \xmark  &\xmark\\
\hline 
%----------------------------------------------
$(2(\sqrt{2} + 1), 4(\sqrt{2} +1))$ & \xmark & \xmark & \xmark&\xmark\\
\hline
$(2(-\sqrt{2} + 1), 4(-\sqrt{2} + 1))$ & \xmark & \xmark &\xmark & \xmark \\
\hline
\end{longtable}  
\end{center}
No new points coming by the second step. Therefore, 
$$S^0_2(F_4,\Q(\alpha))=\{(\delta ,-\alpha^3 + 2 \alpha ), (\sqrt{-1}\delta ,-\alpha^3 + 2 \alpha )
\}.
$$
This finishes the proof for the computation of $\Gamma_2(F_4,\Q(\alpha))$.
\subsection{\texorpdfstring{$\Gamma_2(F_4,\Q(\zeta_8))$}{}}
Then $K=\Q(\z)$ and
$$
\begin{array}{l}
E_1(\Q(\zeta_8))=E_1(\Q(\sqrt{-1}))\cup\{\pm R_{1,1},\pm R_{1,2},\pm R_{1,3},\pm R_{1,4}\},\\
E_2(\Q(\zeta_8))=E_2(\Q)\cup \{ \pm R_{2,1},\pm R_{2,2},\pm R_{2,3},\pm R_{2,4}\},\\
H_3(\Q(\zeta_8))=H_3(\Q(\sqrt{-1}))\cup 
(\pm 2\zeta_8^3 ,\pm 4(\zeta_8^3 + \zeta_8)),(\pm 2\zeta_8 ,\pm 4(\zeta_8^3 + \zeta_8))\}.\end{array}
$$

 {\bf Step I.}   For each value of $s\notin\Q(\sqrt{-1})=\Q(\z^2)$ and $s\notin\Q(\sqrt{2})=\Q(\z^3-\z)$, we deduce the following points in $\Gamma_2(F_4,\Q(\zeta_8))$: 
$$
\renewcommand{\arraystretch}{1.5}
\begin{array}{|c|c|c|}
\hline 
s & (x,y)\in \Gamma_2(F_4,\Q(\zeta_8))) & \text{Minimal polynomial of $b$ over $\Q(\z)$}\\
\hline
2\z& (-b , -\z )& T^2+\sqrt{2}\\ % T^2 + \z^3 - \z\\     % (z^3-z)^2=2
 \hline 
  -2\z^3 - 2\z^2 - 2\z& (-\z , -b )& T^2+\sqrt{2}\\ % T^2 + \z^3 - \z\\     % (z^3-z)^2=2 
 \hline 
 -2\z^3& (-b , -\z^3 )& T^2+\sqrt{2}\\ % T^2 + \z^3 - \z\\     % (z^3-z)^2=2
 \hline 
 2\z^3 + 2\z^2 + 2\z& (-\z^3 , -b )& T^2+\sqrt{2}\\ % T^2 + \z^3 - \z\\     % (z^3-z)^2=2 
 \hline 
 -2\z& (-b , -\z )& T^2-\sqrt{2}\\ % T^2 - \z^3 + \z\\     % (z^3-z)^2=2
 \hline 
 2\z^3 - 2\z^2 + 2\z& (-\z , -b )& T^2-\sqrt{2}\\ % T^2 - \z^3 + \z\\     % (z^3-z)^2=2
 \hline 
 2\z^3& (-b , -\z^3 )& T^2-\sqrt{2}\\ % T^2 - \z^3 + \z\\     % (z^3-z)^2=2
 \hline 
  -2\z^3 + 2\z^2 - 2\z& (-\z^3 , -b )&T^2-\sqrt{2}\\ % T^2 - \z^3 + \z\\     % (z^3-z)^2=2
 \hline 
  2\z^3 - 2\z - 2& (-b , -\z^2b )& T^2 + 1/2(-\z^3 + \z)\\ 
  \hline 
  -2\z^3 + 2\z - 2&(-b , -\z^2b )& T^2 + 1/2(\z^3 - \z)\\ 
 \hline 
\end{array}
$$
We obtain the points $(\z, \sqrt[4]{2}),(\z^3,\sqrt[4]{2}),(\z, \z^2\sqrt[4]{2}),(\z^3, \z^2\sqrt[4]{2})$ and $(1/\sqrt[4]{2},\z^2/\sqrt[4]{2})$ in $S^0_2(F_4,\Q(\zeta_8))$. 

\

\noindent {\bf Step II.} We compute table \ref{step2zeta8}. {
In the next table, for each pair of points $P_j\in E_j(\Q(\z))$, $j=1,2$ such that the corresponding position in the table \ref{step2zeta8} is not a \xmark\, and it is primitive in $\Q(\z)$, we simplify the corresponding point in $S^0_2(F_4,\Q(\zeta_8))$:
$$
\renewcommand{\arraystretch}{1.5}
\begin{array}{|c|c|}
\hline
\begin{array}{cc}
(1/4(-\z^3 + \z)s + 1/2(-\z^3 + \z) , 
1/4(-\z^3 + \z)s + 1/2(\z^3 + \z)) \\
s^2+(2+2\z^2) s-4 \z^2=0
\end{array}
      &  (\zeta_3 \zeta_8^3,-\zeta_3^2\zeta_8^3)\\
  \hline
  \begin{array}{c}  
(1/4(-\z^3 + \z)s + 1/2(-\z^3 + \z) , 
1/4(-\z^3 + \z)s + 1/2(-\z^3 - \z)) \\
s^2 + (-2\z^2 + 2)s + 4\z^2=0
\end{array}
  &  (-\zeta_3^2 \zeta_8,\zeta_3\zeta_8) \\
  \hline
\begin{array}{c}  
(1/4(-\z^3 - \z)s + 1/2(\z^3 + \z) ,
 1/4(-\z^3 + \z)s + 1/2(\z^3 + \z))\\
s^2 + (2\z^2 - 2)s + 4\z^2 =0
 \end{array} & (-\zeta_3 \zeta_8^3,-\zeta_3^2\zeta_8) \\
 \hline
 \begin{array}{c}
(1/4(-\z^3 - \z)s + 1/2(\z^3 + \z) ,
 1/4(-\z^3 + \z)s + 1/2(-\z^3 - \z))\\
s^2 + (-2\z^2 - 2)s - 4\z^2=0
\end{array} &  (-\zeta_3^2 \zeta_8,\zeta_3\zeta_8^3) \\
  \hline
\end{array}  
$$

$$
\renewcommand{\arraystretch}{1.5}
  \begin{array}{|c|c|}
  \hline
\begin{array}{c}
(1/8(-\z^2 + 1)s + 1/4(-\z^2 - 3), 
 1/8(-\z^3 - \z)s + 1/4(3\z^3 + 3\z))\\
s^2 + (-6\z^2 - 6)s + 4\z^2=0
\end{array} & \omega^{-1}(-\z^2,-\overline{\omega}\z)\\
\hline
\begin{array}{c}
(1/8(-\z^2 + 1)s + 1/4(-3\z^2 - 1) ,
 1/8(-\z^3 - \z)s + 1/4(-3\z^3 - 3\z))\\
s^2 + (-6\z^2 + 6)s - 4\z^2=0
\end{array} & \overline{\omega}^{\,-1}(-1,\omega\z^3) \\
\hline
\begin{array}{c}
(-1/2\z s + 1/2\z^3 ,
 -1/2\z^2s + 1/2)\\
s^2 + \z^2s - 2=0
\end{array}  & \omega^{-1}(-\overline{\omega}\z^3,1) \\
\hline
\begin{array}{c}
(-1/4\z^3s + \z , -1/4s )\\ 
s^2 + 2\z^2s- 8=0
\end{array}  & \overline{\omega}^{\,-1}(-\omega\z,\z^2)\\
\hline\hline
\begin{array}{c}
(1/8(-\z^2 - 1)s + 1/4(-\z^2 + 3),  
1/8(-\z^3 - \z)s + 1/4(3\z^3 + 3\z))\\
s^2 + (6\z^2 - 6)s - 4\z^2=0
\end{array} & \overline{\omega}^{\,-1}(-\z^2,-\omega\z^3) \\
\hline
\begin{array}{c}
(1/8(-\z^2 - 1)s + 1/4(-3\z^2 + 1) ,
 1/8(-\z^3 - \z)s + 1/4(-3\z^3 - 3\z))\\
 s^2 + (6\z^2 + 6)s + 4\z^2=0
\end{array}& \omega^{-1}(1,\overline{\omega}\z)\\
\hline
\begin{array}{c}
(-1/4\z s + \z^3 , -1/4s )\\
s^2 - 2\z^2s - 8=0
\end{array} & \omega^{-1}(-\overline{\omega}\z^3,-\z^2)\\
\hline
\begin{array}{c}
(-1/2 \z^3 s + 1/2 \z , -1/2 \z^2 s - 1/2 )	\\ 
s^2 - \z^2 s - 2=0
\end{array}     &  \omega^{-1}(-\overline{\omega}\z,-1)\\
  \hline
\end{array}
$$
We obtain the following points in  $S^0_2(F_4,\Q(\zeta_8))$: 
$$
\begin{array}{cc}
(\zeta_3\zeta_8,\zeta_3^2\zeta_8),\quad(\zeta_3\zeta_8^3,\zeta_3^2\zeta_8^3),\quad
(\zeta_3\zeta_8,\zeta_3^2\zeta_8^3),\quad
(\zeta_3\zeta_8^3,\zeta_3^2\zeta_8),\\[2mm]
\omega^{-1}(\overline{\omega}\z,1),\quad 
\omega^{-1}(\overline{\omega}\z^3,1),\quad 
\omega^{-1}(\overline{\omega}\z,\z^2),\quad
\omega^{-1}(\overline{\omega}\z^3,\z^2) ,
\\[2mm]
\overline{\omega}^{\,-1}({\omega}\z,1),\quad 
\overline{\omega}^{\,-1}({\omega}\z^3,1),\quad
\overline{\omega}^{\,-1}({\omega}\z,\z^2),\quad
\overline{\omega}^{\,-1}({\omega}\z^3,\z^2).
\end{array}
$$
This finishes the computation of $S^0_2(F_4,\Q(\zeta_8))$ and in particular of $\Gamma_2(F_4,\Q(\zeta_8))$. 

\subsection*{Proof of Theorem \texorpdfstring{\ref{theorem} \eqref{t4}}{}}
It is a straightforward consequence of the method described in Section~\ref{algo} for computing \( \Gamma_2(F_4, K) \) via the sets \( E_i(K) = E_i(L) \) for \( i = 1, 2 \).

\begin{landscape}
\thispagestyle{empty} 
\setcounter{table}{0}
\footnotesize
\setlength\LTleft{-26mm}
\renewcommand{\arraystretch}{1.5}
\begin{longtable}{|c|c|c|c|c|c|c|c|c|c|}
\caption{Step II: $\Q(\z)$}\label{step2zeta8}\\
\hline
\backslashbox{$P_2$}{$P_1$} & $(0,0)$ & $(2,4)$ & $(-2, 4 \sqrt{-1})$ & $(2 \sqrt{-1}, 0)$ & $(-2 \sqrt{-1}, 0)$ & $R_{1,1}$ & $R_{1,2}$ & $R_{1,3}$ & $R_{1,4}$\\
\hline 
%----------------------------------------------
$(0,0)$ & \xmark & \xmark & \xmark & \xmark & \xmark & \xmark& \xmark& \xmark & \xmark \\
\hline 
%----------------------------------------------
$(2,0)$ & \xmark
& $\begin{array}{c} (-s - 1 , -s)\\ x^2+x+2  \end{array}$  &  
$\begin{array}{c} (-1/2\z^2s , -1/2\z^2s + \z^2) \\ x^2-2 x+8 \end{array}$ &
$\begin{array}{c} (1/4(-\z^3 + \z)s + 1/2(-\z^3 + \z) , \\
1/4(-\z^3 + \z)s + 1/2(\z^3 + \z)) \\
x^2+(2+2\z^2) x-4 \z^2\end{array}$  &
$
\begin{array}{c}  
(1/4(-\z^3 + \z)s + 1/2(-\z^3 + \z) , \\
1/4(-\z^3 + \z)s + 1/2(-\z^3 - \z)) \\
x^2 + (-2\z^2 + 2)x + 4\z^2
\end{array}$& \xmark& \xmark& \xmark& \xmark \\
\hline
%----------------------------------------------
$(-2,0)$ &  \xmark
& 
$\begin{array}{c}  (-1/2\z , -1/2\z^2s - \z^2)\\ x^2+2 x+8 \end{array}$ 
&
$\begin{array}{c} (-\z^2s + \z^2 , -s)\\ x^2- x+2 \end{array}$ 
&
$\begin{array}{c}  
(1/4(-\z^3 - \z)s + 1/2(\z^3 + \z) ,\\
 1/4(-\z^3 + \z)s + 1/2(\z^3 + \z))\\
x^2 + (2\z^2 - 2)x + 4\z^2
 \end{array}$ 
 &
$\begin{array}{c}
(1/4(-\z^3 - \z)s + 1/2(\z^3 + \z) ,\\
 1/4(-\z^3 + \z)s + 1/2(-\z^3 - \z))\\
x^2 + (-2\z^2 - 2)x - 4\z^2
\end{array}$ &
\xmark& \xmark& \xmark& \xmark \\
 \hline
%----------------------------------------------
$R_{2,1}$ & \xmark& 
$\begin{array}{c}
(1/8(-\z^2 + 1)s + 1/4(-\z^2 - 3), \\
 1/8(-\z^3 - \z)s + 1/4(3\z^3 + 3\z))\\
x^2 + (-6\z^2 - 6)x + 4\z^2
\end{array}$
& 
$\begin{array}{c}
(1/8(-\z^2 + 1)s + 1/4(-3\z^2 - 1) ,\\
 1/8(-\z^3 - \z)s + 1/4(-3\z^3 - 3\z))\\
x^2 + (-6\z^2 + 6)x - 4\z^2
\end{array}$
&
$\begin{array}{c}
(-1/2\z s + 1/2\z^3 ,
 -1/2\z^2s + 1/2)\\
x^2 + \z^2x - 2
\end{array}$&
$\begin{array}{c}
(-1/4\z^3s + \z , -1/4s )\\ 
x^2 + 2\z^2x - 8
\end{array}$& \xmark& \xmark& \xmark& \xmark \\
 \hline
%----------------------------------------------
$R_{2,2}$ & \xmark &
$\begin{array}{c}
(1/8(-\z^2 - 1)s + 1/4(-\z^2 + 3), \\ 
1/8(-\z^3 - \z)s + 1/4(3\z^3 + 3\z))\\
x^2 + (6\z^2 - 6)x - 4\z^2
\end{array}$&
$\begin{array}{c}
(1/8(-\z^2 - 1)s + 1/4(-3\z^2 + 1) ,\\
 1/8(-\z^3 - \z)s + 1/4(-3\z^3 - 3\z))\\
 x^2 + (6\z^2 + 6)x + 4\z^2
\end{array}$&
$\begin{array}{c}
(-1/4\z s + \z^3 , -1/4s )\\
x^2 - 2\z^2x - 8
\end{array}$&
$\begin{array}{c}
(-1/2 \z^3 s + 1/2 \z , -1/2 \z^2 s - 1/2 )	\\ 
x^2 - \z^2 x - 2
\end{array}$&
\xmark& \xmark& \xmark& \xmark \\
 \hline
%----------------------------------------------
%----------------------------------------------
$R_{2,3}$ & \xmark & \xmark& \xmark& \xmark &\xmark & \xmark& \xmark& \xmark& \xmark \\
 \hline
%----------------------------------------------
$R_{2,4}$ & \xmark & \xmark& \xmark& \xmark &\xmark & \xmark& \xmark& \xmark& \xmark \\
 \hline
 %----------------------------------------------
$R_{2,5}$ & \xmark & \xmark& \xmark& \xmark &\xmark & \xmark& \xmark& \xmark& \xmark \\
 \hline
 %----------------------------------------------
$R_{2,6}$ & \xmark & \xmark& \xmark& \xmark &\xmark & \xmark& \xmark& \xmark& \xmark \\
 \hline
%----------------------------------------------
\end{longtable}  
\end{landscape}
}

\section*{Appendix}
In Section \ref{algo} a procedure was presented that, given a number field \( K \) such that  \( E_1(K) \) and \( E_2(K) \) are finite, allows for the determination of \( \Gamma_2(F_4,K) \). Therefore, this reduces to determine the growth of the torsion of \( E_1 \) and \( E_2 \) over number fields. Note that $E_2$ is the $-1$-twist of $E_1$. In particular, they are isomorphic over $\Q(\zeta_8)$. Let
$$
\begin{array}{lcl}
R_{1,1}=(2 (\zeta_8^3 +  \zeta_8^2 +  \zeta_8), 4 (\zeta_8^3 +  \zeta_8^2 - 1)), & & R_{2,1}=(2\zeta_8^2, 4\zeta_8^3)),\\
R_{1,2}=(2(- \zeta_8^3 -  \zeta_8^2 -  \zeta_8), 4(\zeta_8^2 +  \zeta_8 + 1)), & & R_{2,2}=(-2\zeta_8^2,  4\zeta_8),\\
R_{1,3}=(2(- \zeta_8^3 +  \zeta_8^2 -  \zeta_8), 4( \zeta_8^3 -  \zeta_8^2 + 1)), & & R_{2,3}=(2(\zeta_8^3 - \zeta_8 - 1), 4(\zeta_8^3 + \zeta_8^2 + \zeta_8)),\\
R_{1,4}=(2( \zeta_8^3 -  \zeta_8^2 +  \zeta_8), 4( \zeta_8^2 -  \zeta_8 + 1)),  & & R_{2,4}=(2(-\zeta_8^3 + \zeta_8 - 1), 4(\zeta_8^3 - \zeta_8^2 + \zeta_8)),\\
S_{1,1}=(2 (\alpha ^3 +  \alpha ^2 -  \alpha) , 4( \alpha ^3 +  \alpha ^2 + 1)), & & R_{2,5}=(2(\zeta_8^3 - \zeta_8 + 1), 4(\zeta_8^3 - \zeta_8 + 1)),\\
S_{1,2}=(2(- \alpha ^3 +  \alpha ^2 +  \alpha) , 4( \alpha ^3 -  \alpha ^2 -1)), & & R_{2,6}=(-2(\zeta_8^3 + \zeta_8 + 1), 4(\zeta_8^3 - \zeta_8 - 1)).\\
\end{array}
$$
The following table illustrates the growth of the torsion of the elliptic curve $E_1$ in cases where the degree is less than 8. This information has been obtained from\footnote{Torsion growth data was computed by the author and Filip Najman\cite{GJN2}.} the LMFDB database \cite{lmfdb}, considering that the LMFDB label for \( E_1 \) is \href{https://www.lmfdb.org/EllipticCurve/Q/32/a/4}{\texttt{32.a4}}.
In each row, the 
second column displays the structure of the torsion over the number field $K$ listed in the first column. The third column shows the points in $E_1(K)_\text{tors}$ not coming from $E_1(K')_\text{tors}$ for any proper subfield $K'$ of $K$. On the right side of the table, we display the diagram showing all the number fields (of degree $<8$) involved in the torsion growth of $E_1$. Note that $\alpha$ satisfies $\alpha^4-2\alpha^2-1=0$.
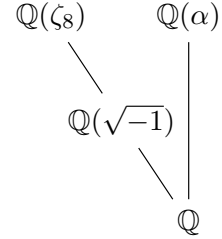
\begin{center}
\begin{tabular}{ccc}
\renewcommand{\arraystretch}{1.5}
\begin{tabular}{|c|c|c|}
\hline
$K$ & $E_1(K)_{\text{tors}}$ & Data $E_1(K)_{\text{tors}}$ \\
 \hline 
$\Q$ & $\Z/4\Z$ & $(0, 0)$, $\pm (2, 4)$\\
\hline
$\Q(\sqrt{-1})$ & $\Z/2\Z\oplus\Z/4\Z$ & $\pm (-2, -4 \sqrt{-1}),
(\pm 2 \sqrt{-1}, 0)$\\
\hline 
$\Q(\zeta_8)$ & $\Z/4\Z\oplus\Z/4\Z$ & $
\pm R_{1,1},
\pm R_{1,2},
\pm R_{1,3},
\pm R_{1,4}$\\
\hline 
$\Q(\alpha)$ & $\Z/8\Z$ & $\pm S_{1,1},\pm S_{1,2}$\\
\hline
\end{tabular} 
& &
{
\begin{tikzcd}[column sep=-1.5 em]
 \Q(\zeta_8)  & &  \Q(\alpha) \\ 
  & \Q(\sqrt{-1}) \arrow[lu, no head] \arrow[rd, no head]  &   \\
 & &  \Q  \arrow[uu, no head]  \\
 \end{tikzcd}
 }
 \end{tabular}
\end{center}

\

Similarly, for the elliptic curve $E_2$, considering that its LMFDB label is \href{https://www.lmfdb.org/EllipticCurve/Q/64/a/3}{\texttt{64.a3}}.

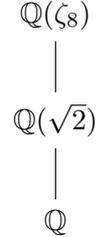
\begin{center}
\begin{tabular}{ccc}
\renewcommand{\arraystretch}{1.5}
\begin{tabular}{|c|c|c|}
\hline
$K$ & $E_2(K)_{\text{tors}}$ & Data $E_2(K)_{\text{tors}}$ \\
 \hline 
$\Q$ & $\Z/2\Z\oplus\Z/2\Z$ & $(0,0),\, (\pm 2, 0)$ \\
\hline 
$\Q(\sqrt{2})$ & $\Z/2\Z\oplus\Z/4\Z$ &
$
\begin{array}{cc}
\pm (2(\sqrt{2} + 1), 4(\sqrt{2} +1)) \\
\pm (2(-\sqrt{2} + 1), 4(\sqrt{2} - 1))
\end{array}$
\\
\hline 
$\Q(\zeta_8)$ & $\Z/4\Z\oplus\Z/4\Z$ & $
\pm R_{2,1},
\pm R_{2,2},
\pm R_{2,3},
\pm R_{2,4},
\pm R_{2,5},
\pm  R_{2,6}$\\
\hline 
\end{tabular} 
& &
{
\begin{tikzcd}
\Q(\zeta_8)  \\ 
\Q(\sqrt{2}) \arrow[u, no head] \arrow[d, no head]   \\
 \Q    \\
 \end{tikzcd}
 }
 \end{tabular}
\end{center}

The following diagram illustrates the lattice of number fields associated with the growth of the torsion of the elliptic curves $E_1$ and $E_2$:

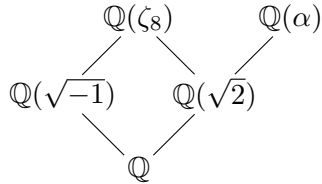
\begin{center}
\begin{tikzpicture}[
    node distance=1.3cm and 1.3cm, 
    every node/.style={inner sep=1pt}, 
    every path/.style={draw} 
    ]
    \node (Q) at (0,0) {$\mathbb{Q}$};  
    \node (Qi)  at (-1,1) {$\mathbb{Q}(\sqrt{-1})$};
    \node (Qz8) at (0,2) {$\mathbb{Q}(\zeta_8)$}; 
   \node (Q2)  at (1,1) {$\mathbb{Q}(\sqrt{2})$};
          \node (Qalpha) at (2,2) {$\mathbb{Q}(\alpha)$}; 
    \draw (Q) -- (Q2);
    \draw (Q) -- (Qi);    
    \draw (Qi) -- (Qz8);
    \draw (Q2) -- (Qz8);
     \draw (Q2) -- (Qalpha);
\end{tikzpicture}
\end{center}

 \subsection*{Some computations}

In this section, we show some number fields that fulfill the conditions required for the application of Theorem \ref{theorem}.
 
\

The following list shows all the square-free integers $D\ne -1,2$, $|D|<200$, satisfying condition $\operatorname{rank}_{\Z}E_1(\Q(\sqrt{D}))=\operatorname{rank}_{\Z}E_2(\Q(\sqrt{D})=0$:
$$
-2,\pm 33, \pm 57, \pm 66, \pm 73, \pm 89, \pm 114, \pm 129, \pm 146, \pm 177, \pm 178, \pm 185.
$$

In the following table we show the defining polynomial of five number fields $K$ such that $\operatorname{rank}_{\Z}E_1(K)=\operatorname{rank}_{\Z}E_2(K)=0$ for $[K:\Q]=d$ such that $d=3,4,5$. 

\renewcommand{\arraystretch}{1.5}
\begin{longtable}{|c|c|c|}
\hline
$[K:\Q]=3$ & $[K:\Q]=4$ & $[K:\Q]=5$   \\
 \hline 
\endfirsthead
\multicolumn{3}{c}%
{{\bfseries \tablename\ \thetable{} -- continued from previous page}} \\
\hline
$[K:\Q]=3$ & $[K:\Q]=4$ & $[K:\Q]=5$   \\
 \hline 
\endhead
\hline \multicolumn{3}{|r|}{{Continued on next page}} \\ \hline
\endfoot
\endlastfoot
$x^3 - x^2 - 2x + 1$ & $ x^4 + 1$ & $ x^5 + x^3 - 2x^2 - 1$  \\
$ x^3 - 3x - 1$ & $x^4 + x^2 - x + 1$ & $ x^5 - x^4 + x^3 - x^2 + 2x - 1$  \\
$ x^3 - x^2 - 4x - 1$ & $x^4 + x^2 - 2x + 1$ & $x^5 - x^3 + 2x - 1$  \\
$ x^3 - x^2 + 2x + 2$ & $x^4 - 2x^3 + 2$ & $x^5 - x^4 + x^3 - x^2 + 1$  \\
$ x^3 - x^2 - 4x + 1$ & $x^4 - 2x^3 + x^2 - 2x + 1$ & $x^5 - x^4 + x^3 + x - 1$  \\
\hline
\end{longtable}

For degree 6 or 7 cases, we were only able to identify the number field \( \mathbb{Q}(\delta) \) with $\delta$ has minimal polynomial $x^6 - x^4 - 2x^3 + 2x + 1$ that fulfill the conditions required for the application of Theorem \ref{theorem}. In most number fields where we attempted to demonstrate that $\operatorname{rank}_{\Z}E_j(K)=0$, $j=1,2$,
Magma was unable to perform the computation. Consequently, we opted to compute the analytic rank, but this also proved unsuccessful. Our final approach was to determine whether $L(E_1/K, 1)$ and $L(E_2/K, 1)$ vanish or not. However, this method also failed to provide new number fields satisfying the conditions required to apply Theorem \ref{theorem}.

\end{document}